\documentclass[a4paper,11pt]{article}
\usepackage[T1]{fontenc}
\usepackage[bitstream-charter]{mathdesign}
\usepackage[french,english]{babel}
\usepackage[babel=true,kerning=true]{microtype}
\usepackage[utf8]{inputenc} 
\usepackage{graphicx}
\usepackage{amsmath}
\usepackage{amsthm}
\usepackage{geometry}
\usepackage{tikz}
\usepackage{hyperref}

\newcommand{\abs}[1]{\left\lvert #1 \right\rvert}
\newcommand{\Int}{\mathrm{Int}~}
\newcommand{\norm}[1]{\left\lVert #1 \right\rVert}
\newcommand{\setof}[2]{\left\{ #1 ~ \middle\arrowvert ~ #2 \right\}}
\newcommand{\prodscal}[2]{\left\langle #1 ~ \middle\arrowvert ~ #2 \right\rangle}

\DeclareRobustCommand*{\refThmSurfaceAnosov}{\ref*{thmSurfaceAnosov}}
\DeclareRobustCommand*{\refThmBilliardsHyp}{\ref*{thmBilliardsHyp}}

\DeclareRobustCommand*{\refThmRiccatiBilliard}{\ref*{thmRiccatiBilliard}}

\hypersetup{
colorlinks=true,
pdfauthor={Mickaël Kourganoff},
pdftitle={Thèse},
pdfcreator=PDFLaTeX,
pdfproducer=PDFLaTeX,
linkcolor=[rgb]{0,0,0.7},
citecolor=[rgb]{0.7,0,0},
}

\usepackage{url}

\makeatletter
\def\url@leostyle{
 \@ifundefined{selectfont}{\def\UrlFont{\sf}}{\def\UrlFont{\small\ttfamily}}}
\makeatother
\newtheorem{prop}{Proposition}[section]
\newtheorem{thm}[prop]{Theorem}
\newtheorem{lemme}[prop]{Lemma}

\theoremstyle{remark}

\theoremstyle{definition}

\newtheorem{definition}[prop]{Definition}

\usetikzlibrary{shapes.geometric}
\usetikzlibrary{through,calc}
\usetikzlibrary{plotmarks}
\usetikzlibrary{arrows,decorations.markings}

\newcommand{\deriv}[2]{\frac{\partial #1}{\partial #2}}

\tikzset{
    scale plot marks/.is choice,
    scale plot marks/false/.code={
        \def\pgfuseplotmark##1{\pgftransformresetnontranslations\csname pgf@plot@mark@##1\endcsname}
    },
    scale plot marks/true/.style={},
    scale plot marks/.default=true
}

\tikzset{fleche/.style args={#1:#2}{ postaction = decorate,decoration={name=markings,mark=at position #1 with {\arrow[#2,scale=2]{>}}}}}

\tikzstyle{vertex}=[circle, fill=black, inner sep=0pt, minimum size=5pt]
\tikzstyle{edge}=[fill=black!0, scale=0.8]
\tikzstyle{fixed}=[rectangle, fill={rgb:blue,2;black,2}, inner sep=0pt, minimum size=5pt]
\tikzstyle{input}=[regular polygon, regular polygon sides=3, rotate=180, fill={rgb:red,0;green,2;black,2}, inner sep=0pt, minimum size=8pt]
\tikzstyle{output}=[regular polygon, regular polygon sides=3, fill={rgb:red,2;black,2}, inner sep=0pt, minimum size=8pt]

\newcommand{\parallelsum}{\mathbin{\!/\mkern-5mu/\!}}

\begin{document}

\selectlanguage{english}

\title{Uniform hyperbolicity in nonflat billiards}
\author{Mickaël Kourganoff\footnote{Institut Fourier, Grenoble University, France, mickael.kourganoff@univ-grenoble-alpes.fr}}
\date{}

\maketitle

\begin{abstract}
Uniform hyperbolicity is a strong chaotic property which holds, in particular, for Sinai billiards. In this paper, we consider the case of a nonflat billiard, that is, a Riemannian surface with boundary. Each trajectory follows the geodesic flow in the interior of the billiard, and bounces when it meets the boundary. We give a sufficient condition for a nonflat billiard to be uniformly hyperbolic. As a particular case, we obtain a new criterion to show that a closed surface has an Anosov geodesic flow.
\end{abstract}

\section{Introduction and notations}
In this paper, a \emph{smooth billiard} is a connected compact subset $D$ of a Riemannian surface $M$, such that $D$ has a smooth boundary while $M$ has no boundary. By ``smooth boundary'', we mean that each component of $\partial D$ is the image of a smooth embedding $\Gamma : \mathbb R / l \mathbb Z \to M$, with unit speed, where $l$ is the length of the component. Each curve $\Gamma$ is called a \emph{wall} of $D$: it has a unit tangent vector $T$ and a unit normal vector $N$ pointing toward $\Int D$. A billiard whose walls have negative curvature is said to be \emph{dispersing}.

Most of the billiards which appear in the literature are \emph{flat}, and more precisely, in the ambiant surface $M = \mathbb R^2$ or $M = \mathbb T^2$. Chaotic billiards in general Riemannian surfaces were studied, for example, in~\cite{MR807598}, \cite{MR1022522}, and~\cite{MR3661861}. For billiards in surfaces of constant curvature, see also~\cite{MR1422542} and~\cite{MR1729878}. In this paper, we focus on \emph{uniform hyperbolicity} (see Definition~\ref{defHyperbolicBilliard}) for billiards in general surfaces.

One defines the phase space $\Omega = T^1(\mathrm{Int}~D)$, and the billiard flow $\phi_t : \Omega \to \Omega$, in the following way:
\begin{enumerate}
\item As long as it does not hit a wall, the particle follows a geodesic in $M$;
\item When it arrives to the boundary of the billiard, the particle bounces, following the billiard reflection law: the angle between the particle's speed vector and the boundary's tangent line is preserved (Figure~\ref{figReflexion}).
\end{enumerate}

The flow $\phi_t$ is not defined at all times :
\begin{enumerate}
\item It is not defined at times when the particle is on the boundary of the billiard. Of course, one could extend the definition to such $t$, but the flow obtained in this way would not be continuous\footnote{Many authors change the topology of $\Omega$ in order to make the flow continuous, but it cannot be made differentiable.}.
\item When the particle makes a grazing collision with a wall at a time $t_0 > 0$, \emph{i.e.} collides with the boundary with an angle $\theta=0$, the flow stops being defined for all times $t \geq t_0$. Although one could extend continuously the definition of the trajectory after such a collision, the differentiability of the flow would be lost.
\end{enumerate}

\begin{figure}[!ht]
\begin{center}
\begin{tikzpicture}
	\clip (-4,-1) rectangle (4,2.5);
    \path[fill=black!20] (-4,-4) rectangle (4,4);
    \path[draw=black, thick, fill=white] (0,-10) circle (10);
    \draw (1,0) arc (0:45:1);
    \node[] at (20:1.5)  {$\theta$};
    \draw (-1,0) arc (180:135:1);
    \node[] at (160:1.5)  {$\theta$};
    \node[circle, fill=black, inner sep=0pt, minimum size=3pt] (qc) at (0,0) {};
    \path[draw=black, fleche=0.3:black, fleche=0.7:black] (-4,4) -- (0,0) -- (4,4);
    \path[draw=black, thick, dashed] (-4,0) -- (4,0);
\end{tikzpicture}

\end{center}
\caption{The billiard reflection law.} \label{figReflexion}
\end{figure}
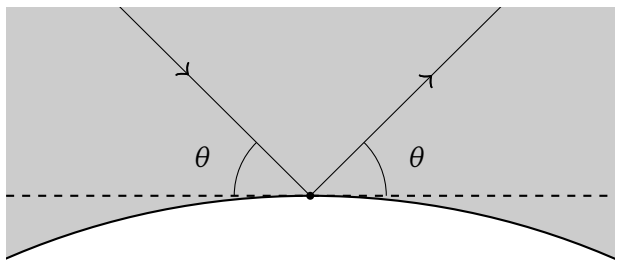

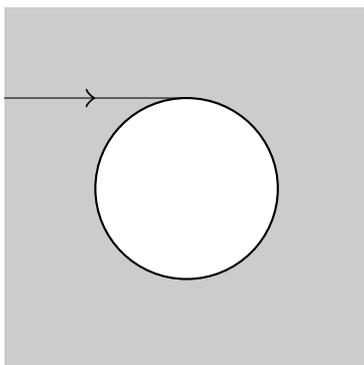
\begin{figure}[!ht]
	\begin{center}
		\begin{tikzpicture}[scale=1.2]
    		\path[fill=black!20] (-2,-2) rectangle (2,2);
    		\path[draw=black, thick, fill=white] (0,0) circle (1);
    		\path[draw=black, fleche=0.5:black] (-2,1) -- (0,1);
		\end{tikzpicture}
\caption{A grazing collision on a dispersing billiard in $\mathbb T^2$. The flow stops being defined after this time.}
	\end{center}
\end{figure}

We define $\tilde \Omega$ as the set of all $(x, v) \in \Omega$ such that the trajectory starting from $(x, v)$ does not contain any grazing collision, in the past or the future. Notice that $\tilde \Omega$ is a residual set of full measure, invariant under the flow $\phi_t$, and that $\phi_t$ is $C^\infty$ on $\tilde \Omega$.

In the special case where $D$ has no boundary, the billiard flow is simply the geodesic flow and $\tilde \Omega = \Omega = T^1 D$.

\bigskip \noindent {\bf Uniform hyperbolicity.} We define uniform hyperbolicity in the case of billiards. This definition is given in a more abstract framework in~\cite{MR2229799}, but here we adapt it directly to billiard flows.
%
%
%
%

\begin{definition} \label{defHyperbolicBilliard}
The billiard flow $\phi_t$ is \emph{uniformly hyperbolic} if at each point $x \in \tilde \Omega$, there exists a decomposition of $T_x\Omega$, invariant under the flow,
\[ T_{x} \Omega = E_x^0 \oplus E_x^u \oplus E_x^s \]
where $E_x^0 = \mathbb R \left. \frac{d}{dt}\right|_{t=0} \phi_t(x)$, such that
\[ \lVert D\phi_x^t|_{E_x^s} \rVert \leq a \lambda^t, \quad \lVert D\phi_x^{-t}|_{E_x^u} \rVert \leq a \lambda^{t} \]
(for some $a > 0$ and $\lambda \in (0,1)$, which do not depend on $x$).
\end{definition}

\bigskip \noindent \textit{Remark.} If the billiard $D$ has no wall (which means that the billiard flow is a geodesic flow), we may use the word \emph{Anosov} instead of \emph{uniformly hyperbolic}.


\section{Results} In this paper, we give a sufficient condition for a (possibly nonflat) billiard to be uniformly hyperbolic.


\subsection{The case of geodesic flows} First, let us consider the case where $D$ has no boundary: the billiard flow is simply the geodesic flow on $D$. All surfaces with negative curvature have an Anosov geodesic flow: according to Arnold and Avez~\cite{MR0209436}, the first proof of this fact goes back to 1898~\cite{hadamard1898surfaces}. Later, it was extended to all manifolds with negative sectional curvature (a modern proof is available in \cite{MR1326374}). But the negative curvature assumption is not necessary for a geodesic flow to be Anosov. To prove that a geodesic flow is Anosov, one may examine the solutions of the Riccati equation
\[ u'(t) = - K(t) - u^2(t) \]
where $K$ is the Gaussian curvature of the surface, and use the following criterion:

\begin{thm} \label{conditionConeRiccati}
Let $M$ be a closed surface. Assume that there exists $t_0 > 0$ such that for any geodesic $\gamma: [0,1] \to M$, and any solution $u$ of the Riccati equation along this geodesic such that $u(0) = 0$, $u$ is well-defined on $[0,t_0]$ and $u(t_0) > 0$. Then the geodesic flow $\phi_t : T^1 M \to T^1 M$ is Anosov.
\end{thm}

Theorem~\ref{conditionConeRiccati} was mentioned in~\cite{donnay2003anosov} and~\cite{magalhaes2013geometry}, without details about the proof. In~\cite{MR3508162}, we apply Theorem~\ref{conditionConeRiccati} to give new examples of surfaces whose geodesic flow is Anosov while their curvature is not negative everywhere. The genus of such surfaces is necessarily at least $2$~\cite{MR0377980}.

In fact, it is possible to improve this theorem by considering an increasing sequence of times $(t_k)_{k \in \mathbb Z} \in \mathbb R^{\mathbb Z}$:
\begin{thm} \label{improvedConditionConeRiccati}
Let $M$ be a closed surface. Assume that there exist $m > 0$ and $C > c > 0$ such that for any geodesic $\gamma: \mathbb R \to M$, there exists an increasing sequence of times $(t_k)_{k \in \mathbb Z} \in \mathbb R^{\mathbb Z}$ with $c \leq t_{k+1} - t_k \leq C$, such that the solution $u$ of the Riccati equation with initial condition $u(t_k) = 0$ is defined on the interval $[t_k, t_{k+1}]$, and $u(t_{k+1}) > m$. Then the geodesic flow $\phi_t : T^1 M \to T^1 M$ is Anosov.
\end{thm}

Notice that Theorem~\ref{conditionConeRiccati} is immediately deduced from Theorem~\ref{improvedConditionConeRiccati} by choosing a constant step $t_{k+1} - t_k$. Theorem~\ref{improvedConditionConeRiccati} is used in~\cite{kourganoff2016embedded} to obtain a surface of genus $12$ embedded in $\mathbb S^3$ with Anosov geodesic flow.

\subsection{The case of billiards} Now we consider the general case, in which $D$ may have a boundary. 

For billiards, we consider a generalized version of the Riccati equation. We say that $u$ is a solution of this equation if:
\begin{enumerate}
\item in the interval between two collisions, $\dot u(t) = -K(t) - u(t)^2$ ;
\item when the particle bounces against the boundary at a time $t$, $u$ undergoes a discontinuity: we have $u(t^+) = u(t^-) - \frac{2 \kappa}{\sin \theta}$, where $\kappa$ is the geodesic curvature of the boundary of $D$, and $\theta$ is the angle of incidence\footnote{The notation $u(t^+)$ stands for $\lim_{h \to 0, h > 0} u(t + h)$, and likewise $u(t^-) = \lim_{h \to 0, h < 0} u(t + h)$. In particular, if $u$ is continuous at $t$, then $u(t^+) = u(t^-) = u(t)$.}.
\end{enumerate}

We are now ready to state the main result of this paper:

\begin{thm} \label{thmRiccatiBilliard}
Consider a (not necessarily flat) billiard $D$. Assume that there exist positive constants $A, m, c$ and $C$ such that for any trajectory $\gamma$ with $\gamma(0) \in \tilde\Omega$, there exists an increasing sequence of times $(t_k)_{k \in \mathbb Z} \in \mathbb R^{\mathbb Z}$ satisfying $c \leq t_{k+1} - t_k \leq C$, such that for any $k \in \mathbb Z$, the solution $u$ of the Riccati equation with initial condition $u(t_k^+) = 0$ satisfies $u(t^+) \geq -A$ for all $t \in [t_k, t_{k+1}]$, and $u(t_{k+1}^+) > m$. Also assume that for each $k \in \mathbb Z$, there is no collision in the interval $(t_k - c, t_k)$, and at most one collision in the interval $(t_k, t_{k+1}]$. Then the billiard flow on $D$ is uniformly hyperbolic.
\end{thm}

Notice that in the particular case where $D$ has no boundary, Theorem~\ref{thmRiccatiBilliard} becomes exactly Theorem~\ref{improvedConditionConeRiccati}. Thus we only need to prove Theorem~\ref{thmRiccatiBilliard}, which will be done in Section~\ref{sectProof}.

\subsection{Applications}

We will explain how Theorem~\ref{thmRiccatiBilliard} can be applied to obtain immediately two famous results: Theorems~\ref{thmSurfaceAnosov} and~\ref{thmBilliardsHyp}. Theorem~\ref{thmRiccatiBilliard} unifies these two theorems, which are both well-known independently. See~\cite{kourganoff2016embedded} for a completely new application of Theorem~\ref{thmRiccatiBilliard}.

\begin{thm} \label{thmSurfaceAnosov}
Let $M$ be a closed Riemannian surface with nonpositive curvature. Assume that every geodesic in $M$ contains a point where the curvature is negative. Then, the geodesic flow on $M$ is Anosov.
\end{thm}

Theorem~\ref{thmSurfaceAnosov} may also be obtained directly, without using Theorem~\ref{conditionConeRiccati}, from Proposition 3.10 of~\cite{MR0380891}. Hunt and MacKay~\cite{MR1986308} used this result to exhibit the first Anosov physical system.

For billiards, we will prove the following counterpart of Theorem~\ref{thmSurfaceAnosov}, which is essentially due to Sinai~\cite{MR0274721}:

\begin{thm} \label{thmBilliardsHyp}
If $D$ is a smooth dispersing flat billiard in $\mathbb T^2$ with \emph{finite horizon}, then the billiard flow is uniformly hyperbolic in~$\tilde \Omega$.
\end{thm}

We say that a billiard has \emph{finite horizon} if every trajectory hits the boundary at least once.

\begin{figure}
	\begin{center}
	\begin{tabular}{cc}
		\begin{tikzpicture}[scale=0.8]
    		\path[fill=black!20] (-2,-2) rectangle (2,2);
    		\path[draw=black, thick, fill=white] (0,0) circle (1);
		\end{tikzpicture}
		&		
		\begin{tikzpicture}[scale=0.8]
    		\clip (-2,-2) rectangle (2,2);
    		\path[fill=black!20] (-1.99,-1.99) rectangle (2,2);
    		\path[draw=black, thick, fill=white] (0,0) circle (0.5);
    		\path[draw=black, thick, fill=white] (2,2) circle (1.8);
    		\path[draw=black, thick, fill=white] (-2,-2) circle (1.8);
    		\path[draw=black, thick, fill=white] (-2,2) circle (1.8);
    		\path[draw=black, thick, fill=white] (2,-2) circle (1.8);
		\end{tikzpicture}
	\end{tabular}
		\caption{On the left, a dispersing billiard in $\mathbb T^2$ with infinite horizon. On the right, a dispersing billiard in $\mathbb T^2$ with finite horizon.}
	\end{center}
\end{figure}
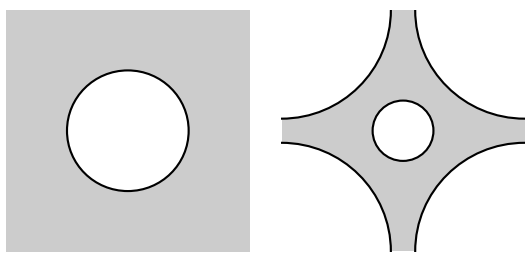

\subsection{Consequences of uniform hyperbolicity} It is shown in~\cite{MR0295390} that (smooth) volume-preserving Anosov flows are ergodic: every invariant subset has either zero or full measure. It was shown later (see~\cite{MR1626749} and~\cite{MR0377980}) that Anosov geodesic flows on surfaces are even exponentially mixing (and then, in all higher dimensions~\cite{MR2113022}).

As for billiard flows, in the flat case only, Sinaï proved ergodicity for smooth dispersing billiards with finite horizon in~\cite{MR0274721}. It was shown in~\cite{baladi2015exponential} that such flows are exponentially mixing.

The consequences of uniform hyperbolicity in the nonflat case are still unknown.

\subsection{Structure of the paper} In Section~\ref{sectCone}, we prove a \emph{cone criterion}, following the ideas of Wojtkowski~\cite{MR782793}. In Section~\ref{sectJacobi}, we study Jacobi fields in (not necessarily flat) billiards. The tools which are introduced in Sections~\ref{sectCone} and~\ref{sectJacobi} are used in Section~\ref{sectProof} to prove Theorem~\ref{thmRiccatiBilliard}. Finally, the two applications are given in Section~\ref{sectAppl}.

\section{The cone criterion} \label{sectCone}
\begin{definition}
Consider a Euclidean space $E$.

A \emph{cone}\footnote{The word ``cone'' has several different meanings in mathematics: here we take the same definition as~\cite{MR1326374}.} in $E$ is a set $C$ such that there exist a decomposition $E = F \oplus G$ and a real number $\alpha \geq 0$ such that

\[ C = \setof{(x, y) \in F \oplus G}{\norm{x} \leq \alpha \norm{y}}. \]

The number $\arctan \alpha$ is called the \emph{angle} of the cone.

Two cones $C_1, C_2$ are said to be \emph{supplementary} if they correspond to decompositions $E = F_1 \oplus G_1$ and $E = F_2 \oplus G_2$ such that $F_1 = G_2$ and $F_2 = G_1$.
\end{definition}

\begin{prop} \label{propCones}
Consider a sequence of invertible linear mappings $A_k: \mathbb R^n \to \mathbb R^n$, $k \in \mathbb Z$, and a sequence of supplementary cones $C_k$ and $D_k$, corresponding to the decomposition $\mathbb R^n = \mathbb R^m \times \mathbb R^{n-m}$.
Assume that there exist $a > 0$, $\lambda > 1$ such that for all $k \in \mathbb Z$:
\begin{enumerate}
\item $A_k (C_k) \subseteq C_{k+1}$ (invariance in the future),
\item \label{expansionFuture} $\norm{A_{k-1} \circ \ldots \circ A_{k-i} (v)} \geq a \lambda^i \norm{v}$ for all $i \geq 0$ and $v \in C_{k-i}$ (expansion in the future),
\item $A_k^{-1} (D_{k+1}) \subseteq D_k$ (invariance in the past),
\item \label{expansionPast} $\norm{A_{k}^{-1} \circ \ldots \circ A_{k+i-1}^{-1} (v)} \geq a \lambda^i \norm{v}$ for all $i \geq 0$ and $v \in D_{k+i}$ (expansion in the past).
\end{enumerate}

Then \[ E_k^u = \bigcap_{i=0}^{+\infty} A_{k-1} \circ \ldots \circ A_{k-i} (C_{k-i}) \] is an $m$-dimensional subspace contained in $C_k$, and \[ E_k^s = \bigcap_{i=0}^{+\infty} A_{k}^{-1} \circ \ldots \circ A_{k+i-1}^{-1} (D_{k+i}) \] is an $(n-m)$-dimensional subspace contained in $D_k$.
\end{prop}

\begin{proof}
For all $i \geq 0$, $A_{k-1} \circ \ldots \circ A_{k-i} (C_{k-i})$ is a cone, which contains a vector space $V_i$ of dimension $m$. Thus, the intersection $E_k^u$ contains a vector space $V$ of dimension $m$ (for example, consider a converging subsequence of orthonormal bases of $V_i$). Assume that there exists $w \in E_k^u \setminus V$. Then there exists $v \in V$ and $t \in \mathbb R$ such that $v + t w \in \{0\} \times \mathbb R^{n-m}$ (notice also that $tw \in E_k^u$). Since $A_{k-i}^{-1} \circ \ldots \circ A_{k-1}^{-1}(tw)$ and $A_{k-i}^{-1} \circ \ldots \circ A_{k-1}^{-1}(v)$ lie in $E_{k-i}^u$, Assumption~\ref{expansionFuture} gives us:

\[ \norm{A_{k-i}^{-1} \circ \ldots \circ A_{k-1}^{-1}(t w)} \leq \frac{1}{a \lambda^i} \norm{t w} \underset{k \to +\infty}{\to} 0, \]
\[ \norm{A_{k-i}^{-1} \circ \ldots \circ A_{k-1}^{-1}(v)} \leq \frac{1}{a \lambda^i} \norm{v} \underset{k \to +\infty}{\to} 0, \]
but at the same time, since $v + t w \in D_k$, Assumption~\ref{expansionPast} gives:
\[ \norm{A_{k-i}^{-1} \circ \ldots \circ A_{k-1}^{-1}(v + t w)} \geq a \lambda^i \norm{v + t w} \underset{k \to +\infty}{\to} +\infty, \]
which contradicts the triangle inequality.

One obtains the result for $E_k^s$ in the same way.
\end{proof}

\begin{thm} \label{thmWoj}
Let $A_k = \begin{pmatrix} a_k & b_k \\ c_k & d_k \end{pmatrix}$ (with $k \in \mathbb Z$) be a sequence of $2 \times 2$ matrices, with determinant $\pm 1$. Fix $\epsilon > 0$, and consider the cone $C_\epsilon$ of all vectors $\begin{pmatrix}x \\ y\end{pmatrix} \in \mathbb R^2$ such that $\epsilon y \leq x \leq \frac{1}{\epsilon} y$. Assume that for all $k$, and all $v = \begin{pmatrix} x \\ y \end{pmatrix}$ with $xy > 0$, \[ A_k v \in C_\epsilon. \]

Then, there exist $a > 0$ and $\lambda > 1$ such that for all $k \in \mathbb Z$, for all $i \geq 0$ and $v \in C_\epsilon$, \[ \norm{A_{k-1} \circ \ldots \circ A_{k-i} (v)} \geq a \lambda^i \norm{v}. \]
\end{thm}

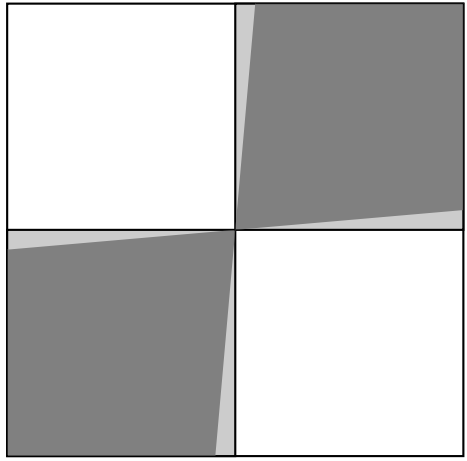
\begin{figure}[!ht]
\begin{center}
\begin{tikzpicture}[auto=right]
\begin{scope}[scale=3]
    \draw[thick] (-1,-1) rectangle (1,1);
    \draw[thick,fill=black!20] (0,0) rectangle (1,1);
    \draw[thick,fill=black!20] (-1,-1) rectangle (0,0);
    \clip (-1,-1) rectangle (1,1);
    \begin{scope}
    \clip (5:2) arc (5:185:2);
    \clip (-95:2) arc (-95:85:2);
    \draw[fill=black!50] (0,0) rectangle (1,1);
    \end{scope}
    \begin{scope}
    \clip (5:-2) arc (5:185:-2);
    \clip (-95:-2) arc (-95:85:-2);
    \draw[fill=black!50] (0,0) rectangle (-1,-1);
    \end{scope}
  \end{scope}
\end{tikzpicture}
\caption{Each $A_k$ maps the cone $xy > 0$ (in grey) into the smaller cone $C_\epsilon$ (in dark grey).}
\end{center}
\end{figure}
\begin{proof}

On the basis of Wojtkowski's idea~\cite{MR782793}, instead of proving expansion directly for the Euclidean norm, we consider the function
\[ \begin{aligned} N: C_\epsilon & \to \mathbb R_{\geq 0}
\\ \begin{pmatrix}x\\y\end{pmatrix} & \mapsto \sqrt{xy}. \end{aligned} \]

Notice that $N$ is equivalent to the Euclidean norm on $C_\epsilon$, \emph{i.e.} there exists $M > 0$ such that for all $v \in C_\epsilon$,
\[ \frac{1}{M} \norm{v} \leq N(v) \leq M \norm{v}, \]
because $\frac{\epsilon}{2} (x^2 + y^2) \leq xy \leq \frac{2}{\epsilon} (x^2 + y^2)$ for all $\begin{pmatrix}x \\y\end{pmatrix} \in C_\epsilon$.

We are going to show that for all $k\in \mathbb Z$ and $v \in C_\epsilon$, $N(A_k v) \geq \frac{1}{1 - \epsilon^2} N(v)$. With the equivalence of norms, this will complete the proof.

Let $k \in \mathbb Z$. We may assume that $\det(A_k) = 1$, by multiplying $A_k$ by $\begin{pmatrix} 0 & 1 \\ 1 & 0 \end{pmatrix}$ on the left. Moreover, we may assume that all the coefficients of $A_k$ are positive, by multiplying $A_k$ by $- \mathrm{Id}$.

Notice that the two vectors $A_k \begin{pmatrix}1 \\ 0 \end{pmatrix} = \begin{pmatrix} a_k \\ c_k \end{pmatrix}$ and $A_k \begin{pmatrix}0 \\ 1 \end{pmatrix} = \begin{pmatrix} b_k \\ d_k \end{pmatrix}$ are in the cone $C_\epsilon$, by continuity of $A_k$.

Then for $v = \begin{pmatrix}x \\ y \end{pmatrix} \in C_\epsilon$:
\[ \begin{aligned} N ( A_k v ) & = (a_k x + b_k y) (c_k x + d_k y)
\\ & \geq (a_k d_k - b_k c_k ) xy + 2 b_k c_k xy
\\ & \geq (1 + 2 b_k c_k) N(v) \end{aligned} \]

But $a_k d_k - b_k c_k = 1$ and $a_k \leq \frac{1}{\epsilon} b_k, d_k \leq \frac{1}{\epsilon} c_k$, so that $b_k c_k \geq \frac{1}{1 - \epsilon^2} - 1$.

Finally, $N(A_k v) \geq \frac{1}{1 - \epsilon^2} N(v)$.
\end{proof}

\section{Jacobi fields} \label{sectJacobi}

\subsection{Jacobi fields for geodesic flows} \label{sectJacobiGeodesic}

The results in this section are standard and will not be proved: see for example~\cite{MR0152974} for details.

Consider a smooth Riemannian surface $(M,g)$. To show that a geodesic flow is hyperbolic, one has to study how the geodesics move away from (or closer to) each other. Thus, one considers small variations of a given geodesic.

\begin{definition}
Consider a geodesic $\gamma: (a, b) \to M$. Consider a geodesic variation of $\gamma$, \emph{i.e.} a smooth function
\[ f(t, s): (a, b) \times (c, d) \to M \]
such that $f(., 0)$ is the geodesic $\gamma$, and for all $s \in (c, d)$, $f(., s)$ is a geodesic.

The vector field $Y = \deriv{f}{s}$ along the curve $\gamma(t)$ is called an \emph{infinitesimal variation of $\gamma$}.
\end{definition}

\begin{prop}
Any infinitesimal variation of $\gamma$ is a solution of the Jacobi equation:
\[ \ddot{Y} = - R(\dot \gamma, Y) \dot \gamma, \]
where $R$ is the Riemann tensor.
The solutions of the Jacobi equation are called \emph{Jacobi fields}.
\end{prop}
%
%

\begin{prop}
Every Jacobi field along a geodesic $\gamma$ is an infinitesimal variation of $\gamma$.
\end{prop}
%
%

We will now be interested in \emph{orthogonal} Jacobi fields:

\begin{lemme} \label{lemmaOrthogonalSmooth}
If $Y(t)$ and $\dot Y(t)$ are orthogonal to $\dot \gamma$ for some $t \in \mathbb R$, then they remain orthogonal for all $t \in \mathbb R$.
\end{lemme}

From now on, assume that $M$ has dimension $2$, that $\gamma$ is a unit speed geodesic, and that $Y$ is a Jacobi field which is orthogonal to $\dot \gamma$. Choose an orientation of the normal bundle of $\gamma$ in $M$ (which has dimension $1$), \emph{i.e.} a vector $e(t) \in T^1_{\gamma(t)}M$ orthogonal to $\dot \gamma(t)$, so that $Y(t)$ is identified by one real coordinate, noted $y(t) = g(Y(t), e(t))$.

The quantity $\dot y$ satisfies
\[ \dot y = \deriv{f}{t} \cdot g(Y, e) = g(\nabla_{\deriv{f}{t}}Y, e) + g(Y, \nabla_{\deriv{f}{t}}e) = g(\nabla_{\deriv{f}{t}}Y, e). \]
Thus:
\[ \dot y = g(\nabla_{\deriv{f}{t}}\deriv{f}{s}, e) = g(\nabla_{\deriv{f}{s}} \dot \gamma, e). \]

In other words, $\dot y$ measures the infinitesimal variation of the vector $\dot \gamma$ with respect to $s$. Thus, when $y$ and $\dot y$ have the same sign, the Jacobi field is diverging: the geodesics go away from each other. When $y$ and $\dot y$ have opposite signs, the Jacobi field is converging. We will consider the ratio $u = \frac{\dot y}{y}$, when it is well-defined (\emph{i.e.} $y \neq 0$), to measure the convergence rate.

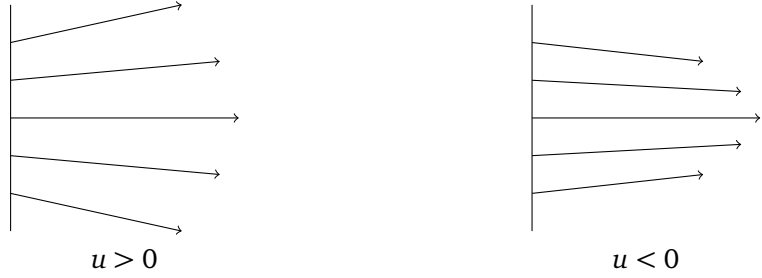
\begin{center}
\begin{tabular}{ccc}
	\begin{tikzpicture}[scale=0.5]
		\draw (0, -3) -- (0,3);
		\draw[->] (0,0) -- (6,0);
		\draw[->] (0,1) -- (5.5,1.5);
		\draw[->] (0,2) -- (4.5,3);
		\draw[->] (0,-1) -- (5.5,-1.5);
		\draw[->] (0,-2) -- (4.5,-3);
	\end{tikzpicture}
	& \hspace{3cm} &
	\begin{tikzpicture}[scale=0.5]
		\draw (0, -3) -- (0,3);
		\draw[->] (0,0) -- (6,0);
		\draw[->] (0,1) -- (5.5,0.7);
		\draw[->] (0,2) -- (4.5,1.5);
		\draw[->] (0,-1) -- (5.5,-0.7);
		\draw[->] (0,-2) -- (4.5,-1.5);
	\end{tikzpicture}
	\\
	$u > 0$ & & $u < 0$
\end{tabular}
\end{center}

\begin{prop}
When it is well-defined, $u$ is a solution of the Riccati equation:
\[ \dot u(t) = - K(\gamma(t)) - u^2(t). \]
where $K$ is the Gaussian curvature.
\end{prop}
%

The solutions of this equation are not always defined for all times: it may happen that $u(t)$ blows up to $-\infty$ in positive time (or to $+\infty$ in negative time). This corresponds to the phenomenon of convergence of the wavefront: up to order $1$, all the geodesics of the infinitesimal variation ``gather at one point''. In most cases, the Jacobi field becomes divergent just after the convergence point (Figure~\ref{figConvergence}).

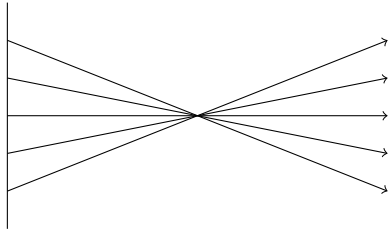
\begin{figure}[!ht]
\begin{center}
	\begin{tikzpicture}[scale=0.5]
		\draw (0, -3) -- (0,3);
		\draw[->] (0,0) -- (10,0);
		\draw[->] (0,1) -- (10,-1);
		\draw[->] (0,2) -- (10,-2);
		\draw[->] (0,-1) -- (10,1);
		\draw[->] (0,-2) -- (10,2);
	\end{tikzpicture}
	\caption{$u$ is not well-defined at the convergence point.} \label{figConvergence}
\end{center}
\end{figure}

From the fact that $u$ satisfies a differential equation, one infers the following \emph{order preserving property}:

\begin{prop} \label{propOrder}
Consider a geodesic $\gamma$ and two Jacobi fields $Y_1(t)$ and $Y_2(t)$ defined on a time interval $[a,b]$. Assume that $y_1(t)$ and $y_2(t)$ do not vanish in this interval (\emph{i.e.} $u_1(t)$ and $u_2(t)$ are well-defined for $t \in [a,b]$). Then $u_2(b) - u_1(b)$ has the same sign as $u_2(a) - u_1(a)$.
\end{prop}

\subsection{Jacobi fields for billiards}

Recall that a smooth billiard $D$ is a compact subset of a Riemannian surface $(M, g)$, such that $D$ has a smooth boundary while $M$ has no boundary. We will write $\prodscal{\cdot}{\cdot}$ for $g(\cdot, \cdot)$.

Consider a billiard trajectory $\gamma$ and a unit speed variation of this trajectory
\[ f(t, s): (a, b) \times (c, d) \to D \]
(defined for all times $t \in (a, b)$, except for the collision times) such that $f(.,0)$ is the trajectory $\gamma$ and for each $s \in (c,d)$, $f(.,s)$ is a billiard trajectory.

By analogy with the case of geodesic flows, we shall call
``Jacobi field'' the vector field $Y = \deriv{f}{s}$ along the curve $\gamma$. 
Inside the billiard, $Y$ satisfies the equation $\ddot Y(t) = K(t) Y(t)$, where $K(t)$ is the curvature at the point $\gamma(t)$ (if the billiard is flat, then $\ddot Y(t) = 0$). At a (non-grazing) collision time, with an angle of incidence $\theta \in (0, \pi/2]$, $Y$ undergoes a discontinuity, which we are now going to study.


Consider a smooth map $s \mapsto \tau(s)$ such that $\tau(s)$ is a collision time of $f(.,s)$ for all $s \in (c, d)$ (reducing the interval $(c, d)$ if necessary). The collision occurs on some component $\Gamma$ of the boundary $\partial D$: assume that $r \mapsto \Gamma(r)$ is a parametrization by arc length and define $r(s)$ so that $\Gamma(r(s))$ is the point where the collision occurs for each $s \in (c, d)$. The parametrization of $\Gamma$ is chosen so that $\prodscal{\dot \gamma (\tau(0)^-)}{ \frac{d\Gamma}{dr}(r(0))} \geq 0$. As in Section~\ref{sectJacobiGeodesic}, choose a section $t \mapsto e_1(t)$ of the unit normal bundle of the trajectory $t \mapsto \gamma(t)$, such that \[ \prodscal{e_1(\tau(0)^-)}{\frac{d\Gamma}{dr}(r(0))} \leq 0 \quad \text{and} \quad \prodscal{e_1(\tau(0)^+)}{\frac{d\Gamma}{dr}(r(0))} \geq 0. \] Define \[ y_\bot(t) = \prodscal{Y(t)}{e_1(t)} \quad \text{and} \quad y_{\parallelsum}(t) = \prodscal{Y(t)}{\dot \gamma(t)}. \]

\begin{prop} \label{propBilliardJacobiParallel}

Writing $y_\bot^\pm = y_\bot(\tau(0)^\pm)$, and defining in the same way $y_{\parallelsum}^\pm$, we have:

\[ y_\bot^+ = - y_\bot^- \quad \text{and} \quad y_{\parallelsum}^+ = y_{\parallelsum}^-.
\]
\end{prop}
\begin{proof}
On the one hand,
\[ \frac{d}{ds} \Gamma(r(s)) = \frac{y_\bot^-}{\sin \theta} \quad \text{and} \quad \frac{d}{ds} \Gamma(r(s)) = - \frac{y_\bot^+}{\sin \theta}, \]
so $y_\bot^+ = - y_\bot^-$.

On the other hand,
\[ \frac{d}{ds} \tau(s) = \frac{- y_\bot^-}{\tan \theta} + y_{\parallelsum}^- \quad \text{and} \quad \frac{d}{ds} \tau(s) = \frac{y_\bot^+}{\tan \theta} + y_{\parallelsum}^+, \]
so $y_{\parallelsum}^+ = y_{\parallelsum}^+.$
\end{proof}

From now on, we consider a \emph{perpendicular Jacobi field}, that is, we assume that $y_{\parallelsum}^- = 0$. Proposition~\ref{propBilliardJacobiParallel} implies that $y_{\parallelsum}^+ = 0$: in other words, any perpendicular Jacobi field remains perpendicular after a collision. We will write $y(t) = y_\bot(t)$ and define $u(t) = \dot y(t) / y(t)$.

\begin{prop}
Assume that the geodesic variation $f$ corresponds to an orthogonal Jacobi field.

At a collision,
\[ \begin{aligned} y^+ &= - y^-
\\ \dot y^+ &= - \dot y^- + \frac{2 \kappa}{\sin \theta} y^-
\\ u^+ &= u^- - \frac{2 \kappa}{\sin \theta} \end{aligned} \]
where $\kappa$ is the curvature of the boundary and $\theta$ is the angle of incidence.

\end{prop}
\begin{proof}
The first equality was already proved in Proposition~\ref{propBilliardJacobiParallel}. To obtain the next equality, consider the billiard reflection law:

\[ \prodscal{\deriv{f}{t} (\tau(s)^+) - \deriv{f}{t} (\tau(s)^-)}{\deriv{\Gamma}{r}(r(s))} = 0. \]

After differentiation with respect to $s$ we obtain:
\[ \prodscal{\dot Y^+ - \dot Y^-}{\deriv{\Gamma}{r}(r(s))} + \prodscal{\deriv{f}{t} (\tau(s)^+) - \deriv{f}{t} (\tau(s)^-)}{\nabla_{\deriv{\Gamma}{r}} \deriv{\Gamma}{r} (r(s)) \cdot \deriv{r}{s}} = 0. \]

We may now compute:
\[ \prodscal{\dot Y^+ - \dot Y^-}{\deriv{\Gamma}{r}(r(s))} = (\dot y^+ + \dot y^-) \sin \theta, \]
\[ \prodscal{\deriv{f}{t} (\tau(s)^+) - \deriv{f}{t} (\tau(s)^-)}{\nabla_{\deriv{\Gamma}{r}} \deriv{\Gamma}{r} (r(s)) \cdot \deriv{r}{s}} = 2 \sin \theta \cdot \kappa \cdot \frac{- y^-}{\sin \theta}. \]

Thus:
\[ \dot y^+ = - \dot y^- + \frac{2 \kappa}{\sin \theta} y^- \]
and since $u = \dot y / y$,
\[ u^+ = u^- - \frac{2 \kappa}{\sin \theta}. \]

\end{proof}

In particular, positively curved walls decrease the value of $u$ (and tend to make the Jacobi field converge), just as the positive curvature of a Riemannian surface. Likewise, negatively curved walls make the quantity $u$ increase, as the negative curvature of a surface.

\section{Proof of Theorem~\refThmRiccatiBilliard} \label{sectProof}

We fix the constants $A$, $c$, $C$ and $m$ which appear in the statement of the theorem, and assume that $A \geq 2$. In this section, we consider times such as $t_a$, $t_b$ or $t_0$, which must not be confused with the times $t_k$ ($k \in \mathbb Z$) which appear in the statement of the theorem.

The readers who are only interested in the proof of Theorem~\ref{improvedConditionConeRiccati} may skip Lemmas~\ref{lemmaUniformContinuity}, \ref{lemmaChangeInitial} and~\ref{lemmaTildes}.

\begin{lemme} \label{lemmaDependence}
Assume that $u$ and $v$ are two solutions of the Riccati equation on an interval $[t_a, t_b]$ with $c/3 \leq t_b - t_a \leq 2C$, such that $0 \leq u(t_a)-v(t_a) \leq \exp (-4AC)$. Assume that $u(t) \geq -A$ for all $t \in [t_a, t_b]$. Then
\[ u(t_b) - v(t_b) \leq (u(t_a)-v(t_a)) \exp (2A(t_b - t_a)). \]
\end{lemme}
\begin{proof}
Let $t_0 = \min \setof{t \in [t_a, t_b]}{u(t) - v(t) \geq 2}$ (with $t_0 = t_b$ if this set is empty).

Then for $t \in [t_a, t_0]$, we have $u(t)-v(t) \geq 0$ (by Proposition~\ref{propOrder}) and
\[ \dot u(t) - \dot v(t) = - (u(t)+v(t))(u(t)-v(t)) \leq 2A(u(t)-v(t)). \]
Thus by Grönwall's lemma
\[ u(t) - v(t) \leq (u(t_a) - v(t_a)) \exp (2A (t - t_a)) \leq 1, \]
so $t_0 = t_b$ and the result is proved.
\end{proof}

From now on we will assume that $m \leq \min (\exp (-4AC), 1/4)$ and define \[ \eta = \min (m^3 / (2K_\mathrm{max} + 2), c/3), \] where $K_\mathrm{max}$ is the maximum absolute value of the curvature on $D$.

\begin{lemme} \label{lemmaUniformContinuity}
Assume that $u$ is a solution of the Riccati equation on an interval $[t_b, t_b + \eta]$, during which no collision occurs. If $\abs{u(t_b)} \leq 1/2$, then \[ \abs{u(t_b + \eta) - u(t_b)} \leq m^3. \]
\end{lemme}

\begin{proof}
Consider $t_0 = \min \setof{t \in [t_b, t_b + \eta]}{\abs{u(t) - u(t_b)} \geq m^3}$ (or $t_0 = t_b + \eta$ if this set is empty). Then for all $t \in (t_b, t_0)$:
\[ \abs{u(t) - u(t_b)} = \abs{ \int_{t_b}^{t} - K(x) - u(x)^2 dx} \leq \eta(K_\mathrm{max} + (\abs{u(t_b)} + m^3)^2) \leq \eta (K_\mathrm{max} + 1) \leq m^3 / 2. \]

This implies that $t_0 = t_b + \eta$ and thus $\abs{u(t_b + \eta) - u(t_b)} \leq m^3$.
\end{proof}

\begin{lemme} \label{lemmaChangeInitial}
Assume that $u$ and $v$ are two solutions of the Riccati equation on an interval $[t_a, t_b]$ with $c/3 \leq t_b - t_a \leq 2C$, with $u(t_a) = 0$ and $v(t_a + \eta) = 0$. Assume that $u(t) \geq -A$ for all $t \in [t_a, t_b]$. Then $v(t_b) \geq u(t_b) - m^2$.
\end{lemme}
\begin{proof}
If $v(t_a) \geq u(t_a)$ then $v(t_b) \geq u(t_b)$ (by Proposition~\ref{propOrder}) and there is nothing to prove. Therefore we assume that $u(t_a) \geq v(t_a)$. Lemma~\ref{lemmaUniformContinuity} implies that $\abs{u(t_a + \eta)} \leq m^3$ and Lemma~\ref{lemmaDependence} shows that $u(t_b) - v(t_b) \leq m^3 \exp (4AC) \leq m^2$.
\end{proof}

From now on, consider a geodesic $\gamma$ and the times $t_k$ given by the assumptions of Theorem~\ref{thmRiccatiBilliard}. For each $k \in \mathbb Z$, define $\tilde t_k$ in the following way:
\begin{itemize}
\item If there is a collision in the interval $[t_k - c/3, t_k]$, define $\tilde t_k = t_k + \eta$.
\item If not, let $\tilde t_k = t_k$.
\end{itemize}

In the following, if $t_k$ is itself a collision time, by $u(t_k)$ we will mean $u(t_k^+)$.

\begin{lemme} \label{lemmaTildes}
For all $k \in \mathbb Z$, the solution $u$ of the Riccati equation with initial condition $u(\tilde t_k) = 0$ satisfies $u(\tilde t_{k+1}) \geq m/2$.
\end{lemme}
\begin{proof}
Consider the solution $v$ of the Riccati equation with initial condition $v(t_k) = 0$.

First, we prove that $u(t_{k+1}) \geq m - m^2$. If $\tilde t_k = t_k$, we have $u = v$ and by assumption $v(t_{k+1}) \geq m$, so $u(t_{k+1}) \geq m$. If $\tilde t_k = t_k + \eta$, Lemma~\ref{lemmaChangeInitial} applied to $v$ and $u$ gives us $u(t_{k+1}) \geq v(t_{k+1}) - m^2 \geq m - m^2$.

Now, we prove that $u(\tilde t_{k+1}) \geq m/2$. If $\tilde t_{k+1} = t_{k+1}$, then $u(\tilde t_{k+1}) = u(t_{k+1}) \geq m - m^2 \geq m/2$. If $\tilde t_{k+1} = t_{k+1} + \eta$, then with Lemma~\ref{lemmaUniformContinuity}, $u(\tilde t_{k+1}) = u(t_{k+1} + \eta) \geq u(t_{k+1}) - m^3 \geq m - m^2 - m^3 \geq m/2$.
%
\end{proof}

\begin{lemme} \label{lemmaReverse}
For all $k \in \mathbb Z$, the solution of the Riccati equation with initial condition $u(\tilde t_{k+1}) = 0$ is well-defined on $[\tilde t_k, \tilde t_{k+1}]$ and satisfies $u(\tilde t_{k}) \leq - m^2/2$.
\end{lemme}
\begin{proof}
Consider the solution $v$ of the Riccati equation with initial condition $v(\tilde t_k) = 0$, and the solution $w$ of the Riccati equation with initial condition $w(\tilde t_k) = -m^2/2$. By Lemma~\ref{lemmaDependence}, $w(\tilde t_{k+1}) \geq v(\tilde t_{k+1}) - (m^2/2) \exp (4AC) \geq v(\tilde t_{k+1}) - m/2$. By Lemma~\ref{lemmaTildes}, $v(\tilde t_{k+1}) \geq m/2$ and thus $w(\tilde t_{k+1}) \geq 0$. 


Now, by Proposition~\ref{propOrder} applied to $u$ and $w$ between the times $\tilde t_k$ and $\tilde t_{k+1}$, the solution of the Riccati equation with initial condition $u(\tilde t_{k+1}) = 0$ satisfies $u(\tilde t_{k}) \leq - m^2/2$. The lemma is proved.
\end{proof}

%

\begin{lemme} \label{lemmaSupCone}
Consider $t_0 \in \mathbb R$ and a solution $u$ of the Riccati equation along a trajectory $\gamma$ defined on the interval $[t_0 - \eta, t_0]$. If $\gamma$ has no collision in the time interval $[t_0 - \eta, t_0]$, then $u(t_0) \leq \alpha$, where \[ \alpha = \sqrt{K_\mathrm{max}} \frac{1 + e^{-2 \sqrt{K_\mathrm{max}} \eta}}{1 - e^{-2 \sqrt{K_\mathrm{max}} \eta}}. \]
\end{lemme}
\begin{proof}
The Riccati equation gives $\dot u(t) \leq K_\mathrm{max} - u(t)^2$.

Notice that whenever $u(t) > \sqrt{K_\mathrm{max}}$, we have $\dot u(t) < 0$. Therefore, the conclusion of the lemma is true if $u(t) \leq \alpha$ for some $t \in [t_0 - \eta, t_0]$.

Now we assume that $u(t) \geq \alpha$ for all $t \in [t_0 - \eta, t_0]$. Thus we may write, for $t \in [t_0 - \eta, t_0]$,
\[ \frac{\dot u(t)}{K_\mathrm{max} - u(t)^2} \geq 1 \]
which implies, after integration between $t_0 - \eta$ and $t_0$:
\[ \frac{u(t_0) - \sqrt{K_\mathrm{max}}}{u(t_0) + \sqrt{K_\mathrm{max}}} \leq e^{-2 \sqrt{K_\mathrm{max}} \eta} \frac{u(t_0 - \eta) - \sqrt{K_\mathrm{max}}}{u(t_0 - \eta) + \sqrt{K_\mathrm{max}}} \leq e^{-2 \sqrt{K_\mathrm{max}} \eta}. \]

Therefore
\[ u(t_0) - \sqrt{K_\mathrm{max}} \leq e^{-2 \sqrt{K_\mathrm{max}} \eta} (u(t_0) + \sqrt{K_\mathrm{max}}) \]
and thus
\[ u(t_0) \leq \alpha. \]
\end{proof}

%
%
%

For each $(x, v) \in \Omega$, the tangent plane $T_{(x, v)} \Omega$ is the direct sum of a vertical and a horizontal subspace $H_{(x, v)} \oplus V_{(x, v)}$, given by the metric $g$ on $M$. Each of these two spaces is naturally endowed with a norm, respectively $g_H$ and $g_V$: one equips $\Omega$ with the norm $g_T = g_H + g_V$ (in particular, one decides that $H$ is orthogonal to $V$).

Denote by $W_{(x, v)} \subseteq T_{(x, v)} \Omega$ the plane orthogonal to the direction of the flow $\phi_t$, and let $(w, w') \in W_{(x,v)}$. There exists $Y(t)$ a Jacobi field such that $(Y(0), \dot Y(0)) = (w, w')$: then the vectors $\dot Y(0)$ and $\dot \gamma (0)$ are orthogonal, and $(Y(t), \dot Y(t)) = D \phi_t(w, w')$ (see \cite{MR1712465} for details). Lemmas~\ref{lemmaOrthogonalSmooth} and~\ref{propBilliardJacobiParallel} imply that $Y(t)$ remains orthogonal to $\dot \gamma(t)$ for all $t$. In particular, the family of planes $(W_{(x, v)})$ (where $(x, v)$ varies in $\tilde \Omega$) is invariant under $D\phi_t$.

Consider an element $(x, v) \in \tilde \Omega$, and $\gamma$ the billiard trajectory such that $(\gamma(0), \dot \gamma(0)) = (x, v)$. Choose an orientation of $H_{(\gamma(t), \dot \gamma(t))} \cap W_{(\gamma(t), \dot \gamma(t))}$, \emph{i.e.} a continuous unit vector $e_1(t)$ in $H_{(\gamma(t), \dot \gamma(t))} \cap W_{(\gamma(t), \dot \gamma(t))}$. It induces naturally an orientation of $V_{(\gamma(t), \dot \gamma(t))}$, given by a continuous unit vector $e_2(t)$ in $V_{(\gamma(t), \dot \gamma(t))}$. This orthogonal basis of $W_{(\gamma(t), \dot \gamma(t))}$ allows us to identify it to the Euclidean $\mathbb R^2$.

For $k \in \mathbb Z$, set
\[ A_k = D_{(\gamma(\tilde t_k), \dot \gamma(\tilde t_k))} \phi_{\tilde t_{k+1} - \tilde t_k} : W_{(\gamma(\tilde t_k), \dot \gamma(\tilde t_k))} \to W_{(\gamma(\tilde t_{k+1}), \dot \gamma(\tilde t_{k+1}))}. \]

The $A_k$ are linear mappings with determinant $\pm 1$, because the flow $\phi_t$ preserves the Liouville measure.

\begin{lemme}
For each $\epsilon > 0$, consider the cones \[ C_\epsilon^\pm = \setof{(x, y) \in \mathbb R^2}{\epsilon y \leq \pm x \leq \frac{1}{\epsilon} y} \quad \text{ and } \quad C_0^\pm = \setof{(x, y) \in \mathbb R^2}{\pm xy > 0}. \]

There exists $\epsilon > 0$ such that for all $k \in \mathbb Z$, \[ A_k C_0^+ \subseteq C_\epsilon^+ \quad \text{ and } \quad A_k^{-1} C_0^- \subseteq C_\epsilon^-. \]
\end{lemme}
\begin{proof}
First, we prove $A_k C_0^+ \subseteq C_\epsilon^+$.

Since the difference between two solutions of the Riccati equation does not change sign, we only need to see that:
\begin{enumerate}
\item The solution of the Riccati equation along $\gamma$ with initial condition $u(\tilde t_k) = 0$ is defined on $[\tilde t_k, \tilde t_{k+1}]$ and satisfies $u(\tilde t_{k+1}) \geq \epsilon$. By Lemma~\ref{lemmaTildes}, it is the case for $\epsilon \leq m/2$.
\item Any solution of the Riccati equation along $\gamma$ with $u(\tilde t_k) \geq 0$ is defined on $[\tilde t_k, \tilde t_{k+1}]$ and satisfies $u(\tilde t_{k+1}) \leq 1/\epsilon$. It is the case for $\epsilon \leq 1/\alpha$, where $\alpha$ is defined in Lemma~\ref{lemmaSupCone}.
\end{enumerate}

Now, let us prove $A_k^{-1} C_0^- \subseteq C_\epsilon^-$. We need to see that:
\begin{enumerate}
\item The solution of the Riccati equation along $\gamma$ with initial condition $u(\tilde t_{k+1}) = 0$ is defined on $[\tilde t_k, \tilde t_{k+1}]$ and satisfies $u(\tilde t_k) \leq - \epsilon$. By Lemma~\ref{lemmaReverse}, it is the case for $\epsilon \leq m^2/2$.

\item Any solution of the Riccati equation along $\gamma$ with $u(\tilde t_{k+1}) \leq 0$ is defined on $[\tilde t_k, \tilde t_{k+1}]$ and satisfies $u(\tilde t_k) \leq 1/\epsilon$. It is the case for $\epsilon \leq 1/\alpha$, where $\alpha$ is defined in Lemma~\ref{lemmaSupCone} (recall that there is no collision in the interval $[\tilde t_{k}, \tilde t_k + \eta]$, according to the assumptions of Theorem~\ref{thmRiccatiBilliard}).
\end{enumerate}
\end{proof}

Thus the sequences $(A_k)$ and $(A_k^{-1})$ satisfy the assumptions of Theorem~\ref{thmWoj}, which provides us with two families of cones: one of them satisfies invariance and expansion in the future, while the other satisfies invariance and expansion in the past. Proposition~\ref{propCones} provides distributions $E^s$ and $E^u$ on $\tilde \Omega$ which are invariant under the flow $\phi_t$, and satisfy
\[ \forall k \in \mathbb Z, \quad \norm{D_{(x,v)}\phi_{t_k}|_{E^s}} \leq a \lambda^k \quad \text{ and } \quad \norm{D_{(x,v)}\phi_{t_{-k}}|_{E^u}} \leq a \lambda^k \] for some $a > 0$ and $\lambda \in (0,1)$.

To go from this discrete statement to a continuous statement, notice the following:

\begin{lemme} \label{lemmaInf}
Consider the set $S$ of all $(t, (x, v)) \in [0,2C] \times T^1 M$ such that the geodesic of length $t$ starting from $(x, v)$ is contained in the billiard $D$.
\[ \sup_{(t, (x, v)) \in S} \norm{D \phi_{t}(x, v)} < +\infty. \]
\end{lemme}
\begin{proof}
The set $S$ is compact.
\end{proof}

Therefore, increasing $a$ and $\lambda$ if necessary, we have:
\[ \forall t \in \mathbb R, \quad \norm{D_{(x,v)}\phi_t|_{E^s}} \leq a \lambda^t \quad \text{ and } \quad \norm{D_{(x,v)}\phi_{-t}|_{E^u}} \leq a \lambda^t. \]

Hence the billiard flow is uniformly hyperbolic, and Theorem~\ref{improvedConditionConeRiccati} is proved.

\section{Applications} \label{sectAppl}

\subsection{Closed surfaces of negative curvature: proof of Theorem~\refThmSurfaceAnosov}

In this proof, we will use the lemma:
\begin{lemme} \label{lemmeSurfaceAnosov}
Under the assumptions of Theorem~\ref{thmSurfaceAnosov}, there exist $m > 0$ and $t_0 > 0$ such that every unit speed geodesic $\gamma: [0,t_0] \to M$ satisfies:
\[ \int_0^{t_0} K(\gamma(t)) dt \leq - m. \]
\end{lemme}
\begin{proof}
If the conclusion is false, consider a sequence $(\gamma_n)$ of unit speed geodesics defined on $[-n, n]$, such that for all $n$,
\[ \int_{-n}^n K(\gamma(t)) dt \geq - \frac{1}{n}. \]
By the Arzelà-Ascoli theorem and a diagonal argument, one may extract a subsequence of $\gamma_n$ which converges uniformly on each $[-n, n]$ to a geodesic defined on $\mathbb R$. By dominated convergence, it satisfies $\int_\mathbb R K(\gamma(t)) dt = 0$, which contradicts the assumption.
\end{proof}

Now, consider the values of $m$ and $t_0$ given by lemma~\ref{lemmeSurfaceAnosov}, and choose a geodesic $\gamma$. We may assume that $m < 1$ and, by dividing the metric of $M$ by a constant if necessary, that $t_0 < 1$.

Denote by $u$ the solution of the Riccati equation $u'(t) = -K(t) - u^2(t)$ with $u(0) = 0$. Since $u'(t) \geq -u^2(t)$, we have $u(t) \geq 0$ for $t \geq 0$, by comparison with the solution $v$ of the differential equation $v'(t) = -v^2(t)$ with initial condition $v(0) = 0$ (here, $v$ is the zero function). In particular, the solution $u$ does not blow up to $-\infty$.

Set $t_1 = \sup \setof{t \in [0,1]}{u(t) \geq m}$ (with $t_1 = 0$ if this set is empty). Thus, for all $t \geq t_1$, \[ u'(t) = -K(t) - u^2(t) \geq - m^2. \]

If $t_1 = 0$, then using the estimate given by Lemma~\ref{lemmeSurfaceAnosov},
\[ u(1) = u(0) + \int_0^1 u'(x) dx = \int_0^1 -K(x) - u^2(x) dx = - \int_0^1 K(x) dx - \int_0^1 u^2(x) dx \geq m - m^2. \]

If $t_1 \neq 0$, then using the fact that $K(t) \leq 0$,
\[ u(1) = u(t_1) + \int_{t_1}^1 u'(x) dx \geq u(t_1) + \int_{t_1}^1 - u^2(x) dx \geq m - m^2. \]

In both cases, one gets $u(1) \geq m - m^2$. One may apply Theorem~\ref{conditionConeRiccati}: the geodesic flow on $M$ is Anosov and Theorem~\ref{thmSurfaceAnosov} is proved.

\subsection{Sinai billiards: proof of Theorem~\refThmBilliardsHyp}

\begin{lemme}
Let $D$ be a flat billiard in $\mathbb T^2$ with finite horizon. Then, there exists $t_0$ such that every billiard trajectory in $\tilde \Omega$ (with unit speed) experiences at least one collision between $t = 0$ and $t = t_0$.
\end{lemme}
\begin{proof}
Assume that the conclusion is false. Then for all $n > 0$, there exists a billiard trajectory $\gamma_n : \mathbb R \to \mathbb T^2$, without collision on $[-n, n]$: we will write $(x_n, v_n) = (\gamma_n(0), \gamma_n'(0))$. Up to extraction, we may assume that $(x_n, v_n)$ has a limit $(x, v) \in \Omega$. The geodesic of $\mathbb T^2$ starting at $(x, v)$ is contained in $D$, so it is periodic (since it cannot be dense in $\mathbb T^2$) with period $T$. If it does not intersect the boundary $\partial D$, then this geodesic is a billiard trajectory without collision, so the billiard does not have finite horizon. Thus, we assume that this geodesic intersects $\partial D$, and since $\partial D$ is smooth, there is an open ball $B_1$ which is tangent to the billiard trajectory, such that $B_1 \cap D = 0$. Furthermore, there is an other ball $B_2$ tangent to the geodesic on the other side, such that $B_2 \cap D = 0$ (otherwise, there is an $x' \in D$ close to $x$ such that the trajectory starting at $(x', v)$ has no collision). If $v_n = v$ for some $n \geq T$, then the trajectory starting at $(x_n, v_n)$ (which has period $T$) has no collision, which again contradicts the finite horizon assumption: thus $v_n \neq v$ for all $n \geq T$. But since $(x_n, v_n)$ tends to $(x, v)$, this implies that there exists $n \geq 2T$ such that the trajectory starting at $(x_n, v_n)$ intersects $B_1$ or $B_2$ in the time interval $[-2T, 2T]$, which contradicts the assumption.
\end{proof}

\begin{lemme}
If $D$ is a flat billiard with finite horizon whose walls have negative curvature, then it satisfies the assumptions of Theorem~\ref{thmRiccatiBilliard}, where the times $t_k$ are the times of collisions.
\end{lemme}
\begin{proof}
We consider the solution $u$ of the generalized Riccati equation, such that $u(t_k^+) = 0$. On the interval $]t_k, t_{k+1}[$, $u$ is a solution of the equation $u'(t) = -u^2(t)$, so $u$ is equal to $0$. Since the walls have positive curvature, $u(t_{k+1}^+) \geq -2 \kappa_{\mathrm{max}} > 0$, where $\kappa_{\mathrm{max}}$ is the maximum curvature of the boundary.
\end{proof}

Thus, Theorem~\ref{thmRiccatiBilliard} applies and concludes the proof.

\section*{Acknowledgements}
This work was supported by the ERC Avanced Grant 320939, Geometry and Topology of Open Manifolds (GETOM).

I would like to thank the referee for helping me improve the paper.

\bibliographystyle{alpha}
\bibliography{ref}
\end{document}